\documentclass[preprint,12pt]{elsarticle}

\usepackage{amssymb,amsthm}
\usepackage{url} 
\usepackage{amsmath}

\newtheorem{thm}{Theorem}

\newtheorem{prp}{Proposition}
\newtheorem{cor}{Corollary}

\newtheorem{exm}{Example} 

\journal{}

\begin{document}

\begin{frontmatter}

\title{Edge metric dimension of some generalized Petersen graphs\tnoteref{msci}}

\author{Vladimir Filipovi\' c \fnref{matf}}
\ead{vladofilipovic@hotmail.com}

\author{Aleksandar Kartelj \fnref{matf}}
\ead{aleksandar.kartelj@gmail.com}

\author{Jozef Kratica \fnref{mi}}
\ead{jkratica@mi.sanu.ac.rs}

\address[mi]{Mathematical Institute, Serbian Academy of Sciences and Arts, Kneza Mihaila 36/III, 11 000 Belgrade, Serbia}
\address[matf]{Faculty of Mathematics, University of Belgrade, Studentski trg 16/IV, 11 000 Belgrade, Serbia} 
 
\tnotetext[msci]{This research was partially supported by Serbian Ministry of
Education, Science and Technological Development under the grants no. 174010 and
174033.}

\begin{abstract}
The edge metric dimension problem was recently introduced, which 
initiated the study of its mathematical properties. The theoretical properties of the edge metric representations and
the edge metric dimension of generalized Petersen graphs $GP(n,k)$ are studied in this paper. 
We prove the exact formulae for $GP(n,1)$ and $GP(n, 2)$, while for the other values of $k$ the lower bound is stated.
\end{abstract}

\begin{keyword}
edge metric dimension, generalized Petersen graphs, discrete mathematics
\end{keyword}

\end{frontmatter}


\section{Introduction}

The concept of metric dimension of a graph $G$ was introduced independently by Slater (1975) in \cite{metd1} 
and Harary and Melter (1976) in \cite{metd2}. This concept is based on the notion of resolving set $R$ of vertices, which has the property that 
each vertex is uniquely identified by its metric representations with respect to $R$. The minimal cardinality 
of resolving sets is called the metric dimension of graph $G$.  

\subsection{Literature review}

Kelenc, Tratnik and Yero (2016) in \cite{kel16} recently introduced a similar concept of edge metric dimension and 
initiated the study of its mathematical properties. They made a comparison between the
edge metric dimension and the standard metric dimension of graphs while presenting realization results concerning the edge metric dimension and the standard metric
dimension of graphs. They also prove that edge metric dimension problem is NP-hard in a general case,
and provided approximation results. Additionally, for several 
classes of graphs, exact values for edge metric dimension were presented, while several others were given upper and lower bounds. 
In \cite{yer16}, the authors presented results of mixed metric dimension alongside the edge metric dimension for some classes of graphs.
Additionally, Peterin and Yero, in \cite{pet18}, provide exact formulas for join, lexicographic and corona product of graphs. 

Zubrilina, in paper \cite{zub16a}, firstly proposed the classification of graphs of $n$ vertices for which the edge metric dimension is equal to its upper bound $n-1$. The second result states that the ratio between edge metric dimension and metric dimension of
an arbitrary graph is not bounded from above. Third result characterize change of the edge dimension of an arbitrary graph upon taking
a Cartesian product with a path, and change of the edge dimension upon adding a vertex adjacent to all the original vertices.
The edge metric dimension of the Erd\"os-R\'enyi random graph $G(n, p)$ is given by Zubrilina in \cite{zub16b}
and it is equal to $(1 + o(1)) \cdot \frac{4 log(n)}{log (1/q)}$, where $q=1-2p(1-p)^2(2-p)$.

Independently, Epstein, Levin and Woeginger (2015) \cite{eps15}, introduced another edge metric dimension definition related to the line graphs. 
Their edge metric dimension of graph $G$ is defined as metric dimension of $L(G)$, 
which is called edge variant of metric dimension by some authors, e.g. Liu et al. (2018) \cite{liu18}.

\subsection{Generalized Petersen graphs}

Generalized Petersen graphs were first studied by Coxeter \cite{gpprvi}. 
Each such graph, denoted as $GP(n,k)$, is defined for $n \geq 3$ and $1 \leq k < n/2$. 
It has $2n$ vertices and $3n$ edges, with vertex set $V(GP(n,k))$ = \{$u_i, v_i \;|\; 0
\leq i \leq n-1\}$ and edge set $E(GP(n,k))$ = \{$u_iu_{i+1},\;u_iv_i, \;v_iv_{i+k}\;|\;0 \leq i \leq n-1$\}.
It should be noted that vertex indices are taken modulo $n$. 

\begin{exm}Consider the Petersen graph, numbered as $GP(5,2)$, shown in Figure 1. It is easily calculated, by using total enumeration technique, that its edge metric dimension is equal to 4 (it is also presented in Table \ref{tabgpn2}). In Figure 2, $GP(6,1)$ with edge metric dimension equal to 3, is presented. This can also be concluded by Theorem \ref{gpn1}.
\end{exm}  

\begin{figure}[htbp]
\centering\setlength\unitlength{1mm}
\begin{picture}(46,48)
\thicklines
\tiny
\put(24.9,40.6){\circle*{2}} \put(24.9,32.6){\circle*{2}}
\put(24.9,40.6){\line(4,-3){19.0}} \put(24.9,40.6){\line(0,-1){8.0}} \put(24.9,32.6){\line(1,-3){7.1}}
\put(23.9,43.6){$u_0$} \put(26.1,34.4){$v_0$}
\put(43.9,26.8){\circle*{2}} \put(36.3,24.3){\circle*{2}}
\put(43.9,26.8){\line(-1,-3){7.3}} \put(43.9,26.8){\line(-3,-1){7.6}} \put(36.3,24.3){\line(-4,-3){18.5}}
\put(45.7,27.7){$u_1$} \put(37.7,22.8){$v_1$}
\put(36.6,4.4){\circle*{2}} \put(31.9,10.9){\circle*{2}}
\put(36.6,4.4){\line(-1,0){23.5}} \put(36.6,4.4){\line(-3,4){4.7}} \put(31.9,10.9){\line(-4,3){18.5}}
\put(37.4,2.0){$u_2$} \put(30.2,8.1){$v_2$}
\put(13.1,4.4){\circle*{2}} \put(17.8,10.9){\circle*{2}}
\put(13.1,4.4){\line(-1,3){7.3}} \put(13.1,4.4){\line(3,4){4.7}} \put(17.8,10.9){\line(1,3){7.1}}
\put(10.4,2.0){$u_3$} \put(14.0,10.7){$v_3$}
\put(5.9,26.8){\circle*{2}} \put(13.5,24.3){\circle*{2}}
\put(5.9,26.8){\line(4,3){19.0}} \put(5.9,26.8){\line(3,-1){7.6}} \put(13.5,24.3){\line(1,0){22.8}}
\put(2.0,27.7){$u_4$} \put(11.4,27.0){$v_4$}
\end{picture}
\caption{Petersen graph GP(5,2)}
\end{figure}
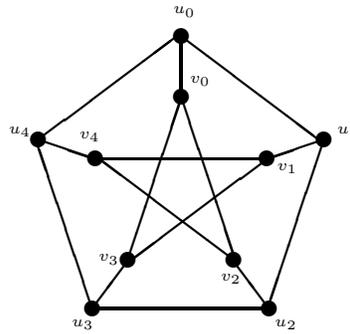

\begin{figure}[htbp]
\centering\setlength\unitlength{1mm}
\begin{picture}(50,44)
\thicklines
\tiny
\put(22.9,45.0){\circle*{2}} \put(22.9,37.0){\circle*{2}}
\put(22.9,45.0){\line(5,-3){17.3}} \put(22.9,45.0){\line(0,-1){8.0}} \put(22.9,37.0){\line(5,-3){10.4}}
\put(21.9,48.0){$u_0$} \put(21.9,34.0){$v_0$}
\put(40.2,35.0){\circle*{2}} \put(33.3,31.0){\circle*{2}}
\put(40.2,35.0){\line(0,-1){20.0}} \put(40.2,35.0){\line(-5,-3){6.9}} \put(33.3,31.0){\line(0,-1){12.0}}
\put(41.8,36.5){$u_1$} \put(29.7,29.5){$v_1$}
\put(40.2,15.0){\circle*{2}} \put(33.3,19.0){\circle*{2}}
\put(40.2,15.0){\line(-5,-3){17.3}} \put(40.2,15.0){\line(-5,3){6.9}} \put(33.3,19.0){\line(-5,-3){10.4}}
\put(41.8,13.5){$u_2$} \put(29.7,20.5){$v_2$}
\put(22.9,5.0){\circle*{2}} \put(22.9,13.0){\circle*{2}}
\put(22.9,5.0){\line(-5,3){17.3}} \put(22.9,5.0){\line(0,1){8.0}} \put(22.9,13.0){\line(-5,3){10.4}}
\put(21.9,2.0){$u_3$} \put(21.9,16.0){$v_3$}
\put(5.6,15.0){\circle*{2}} \put(12.5,19.0){\circle*{2}}
\put(5.6,15.0){\line(0,1){20.0}} \put(5.6,15.0){\line(5,3){6.9}} \put(12.5,19.0){\line(0,1){12.0}}
\put(2.0,13.5){$u_4$} \put(14.1,20.5){$v_4$}
\put(5.6,35.0){\circle*{2}} \put(12.5,31.0){\circle*{2}}
\put(5.6,35.0){\line(5,3){17.3}} \put(5.6,35.0){\line(5,-3){6.9}} \put(12.5,31.0){\line(5,3){10.4}}
\put(2.0,36.5){$u_5$} \put(14.1,29.5){$v_5$}
\end{picture}
\caption{Graph GP(6,1)}
\end{figure}
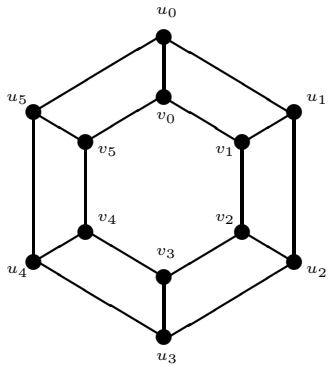

The metric dimension of generalized Petersen graphs $GP(n,k)$ is studied for different values of k:
\begin{itemize}
\item Case $k=1$ is concluded from \cite{cac07};
\item Case $k=2$ is proven in \cite{jav08};
\item Case $k=3$ in \cite{imr14}.
\end{itemize}

Various other properties of generalized Petersen graphs have been recently
theoretically investigated in the following areas: Hamiltonian property \cite{wan17},
the cop number \cite{bal17}, the total coloring \cite{dan16}, etc.

\subsection{Definitions and previous work}

Given a simple connected undirected graph $G = (V,E)$, for $u,v
\in V$ $d(u,v)$ denotes the distance between $u$ and $v$ in $G$,
i.e. the length of a shortest $u-v$ path. A vertex $x$ of the graph
G is said to resolve two vertices $u$ and $v$ of $G$ if $d(x,u) \neq
d(x,v)$. An ordered vertex set $R$ = \{$x_1, x_2, ..., x_k$\} of $G$
is a resolving set of $G$ if every two distinct vertices of $G$ are
resolved by some vertex of $R$. A metric
basis of $G$ is a resolving set of the minimum cardinality. The
metric dimension of $G$, denoted by $\beta(G)$, is the cardinality
of its metric basis.

Similarly, for a given connected graph $G$, a vertex $w \in V$ and an edge $uv \in E$, the distance
between the vertex $w$ and the edge $uv$ is defined as $d(w,uv) = \min\{d(w, u), d(w, v)\}$. A vertex
$w \in V$ resolves two edges $e_1$ and $e_2$ ($e_1,e_2 \in E$), if $d(w, e1) \ne d(w, e2)$.
A set $S$ of vertices in a connected graph $G$ is an edge metric generator for $G$ if every two
edges of $G$ are resolved by some vertex of $S$. The smallest cardinality of an edge metric
generator of $G$ is called the edge metric dimension and is denoted by $\beta_E(G)$. An edge
metric basis for $G$ is an edge metric generator of $G$ with cardinality $\beta_E(G)$.
Given an edge $e \in E$ and an ordered vertex set $S$ = \{$x_1, x_2, ..., x_k$\}, the $k$-touple
$r(e,S)$ = ($d(e,x_1), d(e,x_2), ..., d(e,x_k)$) is called the
edge metric representation of $e$ with respect to $S$.

\begin{exm}Consider the generalized Petersen graph $GP(6,1)$ given on Figure 2.
The set $S_1=\{u_0,  u_1, u_3\}$ is an edge metric generator for $G$ since
the vectors of metric coordinates for edges of $G$ with respect
to $S_1$ are mutually different: $r(u_0u_1,S_1)$=(0,0,2); $r(u_0u_5,S_1)$=(0,1,2);
$r(u_0v_0,S_1)$=(0,1,3); $r(u_1u_2,S_1)$=(1,0,1); 
$r(u_1v_1,S_1)$=(1,0,2); $r(u_2u_3,S_1)$=(2,1,0);
$r(u_2v_2,S_1)$=(2,1,1); $r(u_3u_4,S_1)$=(2,2,0);
$r(u_3v_3,S_1)$=(3,2,0); $r(u_4u_5,S_1)$=(1,2,1);
$r(u_4v_4,S_1)$=(2,3,1); $r(u_5v_5,S_1)$=(1,2,2);
$r(v_0v_1,S_1)$=(1,1,3); $r(v_0v_5,S_1)$=(1,2,3);
$r(v_1v_2,S_1)$=(2,1,2); $r(v_2v_3,S_1)$=(3,2,1);
$r(v_3v_4,S_1)$=(3,3,1); $r(v_4v_5,S_1)$=(2,3,2).
From Corollary \ref{deg}, it holds that for $GP(6,1)$,
as for any other generalized Petersen graph, 
cardinality of edge metric generator must be at least 3,
so $S_1$ is an edge metric basis for $GP(6,1)$. This implies that its edge metric dimension is equal to 3, i.e.
$\beta_E(GP(6,1)) = 3$.
\end{exm}

Two edges are called an incident, if both contain one common endpoint.
For a given vertex $v \in V$, its degree $deg_v$ is equal to number of
its neighbors, i.e. number of edges in which it is endpoint.
Maximum and minimum degree over all vertices of graph $G$ is noted as  
$\Delta(G)$ and $\delta(G)$, respectively. Formally,
$\Delta(G) = \max \limits_{v \in V} \,{\deg _v}$
and $\delta(G) = \min \limits_{v \in V} \,{\deg _v}$.

\begin{prp} \mbox{\rm(\cite{kel16})} \label{edim1} For $n \ge 2$ it holds $\beta_E(P_n) = \beta(P_n) = 1$,
$\beta_E(C_n) = \beta(C_n) = 2$, $\beta_E(K_n) = \beta(K_n) = n-1$. Moreover, $\beta_E(G)=1$ if and only if $G$ is a path $P_n$.
\end{prp}

\begin{prp} \mbox{\rm(\cite{kel16})} \label{edim2} Let $G$ be connected graph and let $\Delta(G)$ be the maximum degree of $G$. Then $\beta_E(G) \geq log_2 \Delta(G)$
\end{prp}

\begin{prp} \mbox{\rm(\cite{kel16})} \label{edim3} Let $G$ be a connected graph and let $S$ be an edge metric basis with $|S| = k$.
Then $S$ does not contain a vertex with the degree greater than $2^{k\-1}$.
\end{prp}

From each of these propositions, given in \cite{kel16}, it follows the next two statements: 

\begin{cor} \label{edim4} Edge metric dimension of the 3-regular graph is at least 2.\end{cor}

\begin{cor} \label{edim5} $\beta_E(GP(n,k)) \geq 2$.\end{cor}

As previosly mentioned, Epstein, Levin and Woeginger in 2015 \cite{eps15} introduced another edge metric dimension definition based on line graphs. 
Line graph of a graph $G(V,E)$ is defined as: $L(G)=(E,F)$ where $F=\{e_{i}e_{j} | e_i, e_j \in E, e_i \; \text{is incident with} \; e_j \}$. 
Their edge metric dimension of graph $G$ is defined as metric dimension of $L(G)$. 
In order to avoid misunderstanding, $\beta_E'(G)$ will denote this second definition of edge metric dimension of graph $G$, also called edge version of metric dimension \cite{liu18}, i.e. $\beta_E'(G)=\beta(L(G)$. Based on this definition, in \cite{nas18}, the authors obtained the results for $n$-sunlet graphs and prism graphs.

Difference between these definitions can be demonstrated with the following example. 
Let $G_1=(V_1,E_1)$ be the graph with $V_1=\{v_0, v_1, v_2, v_3\}$ and $E_1=\{e_0, e_1, e_2, e_3, e_4\}$ such that 
$e_0=v_0v_1, \; e_1=v_1v_2, \; e_2=v_0v_2, \; e_3=v_1v_3, \; e_4=v_2v_3$. Line graph of $G_1$ is $L(G_1)=(E_1, F_1)$ where\\ $F_1=\{e_0e_1, e_0e_2, e_0e_3, e_1e_2, e_1e_3, e_1e_4, e_2e_4, e_3e_4\}$. Graphs $G_1$ and $L(G_1)$ are presented in Figure 3. 

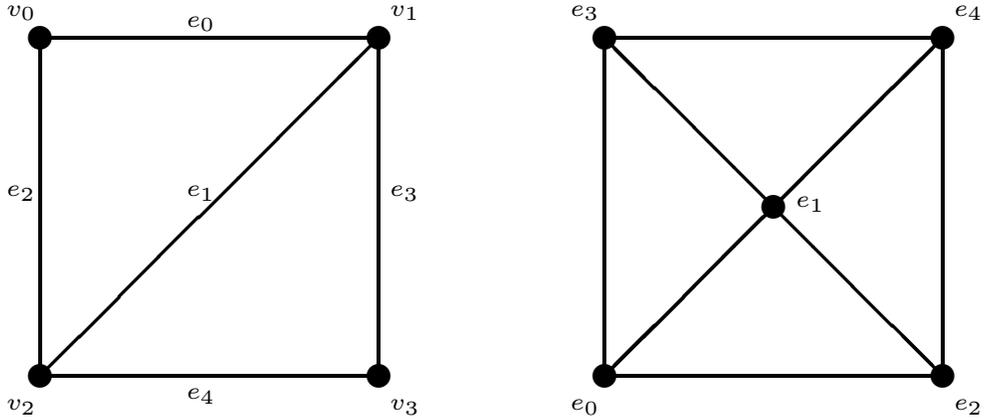
\begin{figure}[htbp]
\centering\setlength\unitlength{1mm}
\scalebox{1.5}{
\begin{picture}(90,40)
\thicklines
\tiny


\put(5,35){\circle*{2}} 
\put(35,35){\circle*{2}}
\put(5,5){\circle*{2}}
\put(35,5){\circle*{2}}

\put(5,5){\line(1,0){30}}
\put(5,5){\line(0,1){30}}
\put(5,35){\line(1,0){30}}
\put(5,5){\line(1,1){30}} 
\put(35,5){\line(0,1){30}}
\put(2,37){$v_0$} 
\put(36,37){$v_1$} 
\put(2,2){$v_2$} 
\put(36,2){$v_3$} 

\put(18,36){$e_0$} 
\put(18,21){$e_1$} 
\put(2,21){$e_2$} 
\put(36,21){$e_3$} 
\put(18,3){$e_4$} 


\put(55,35){\circle*{2}} 
\put(85,35){\circle*{2}}
\put(55,5){\circle*{2}}
\put(85,5){\circle*{2}}
\put(70,20){\circle*{2}}

\put(55,5){\line(1,0){30}}
\put(55,5){\line(0,1){30}}
\put(55,35){\line(1,0){30}}
\put(55,5){\line(1,1){30}} 
\put(85,5){\line(0,1){30}}

\put(70,20){\line(-1,-1){15}}
\put(70,20){\line(-1,1){15}}
\put(70,20){\line(1,-1){15}}
\put(70,20){\line(1,1){15}}

\put(52,37){$e_3$} 
\put(86,37){$e_4$} 
\put(52,2){$e_0$} 
\put(86,2){$e_2$} 
\put(72,20){$e_1$} 

\end{picture}
}
\caption{Graph from Example 1 and its corresponding line graph}
\end{figure}

By using total enumeration technique, it can be shown that $\beta_E(G_1)=3$ with edge metric  base $\{v_0, v_1, v_2\}$. 
On the other hand, $\beta_E'(G_1)=\beta(L(G_1))=2$ with metric base $\{e_0, e_2\}=\{v_0v_1, v_0v_2\}$.

\section{Main results}

\subsection{Lower bound}

Having in mind the fact that vertices from an edge metric basis are also endpoints for some (incident) edges,
the bound presented in Proposition \ref{edim2}, could be improved in some cases.

\begin{thm} Let $G$ be a connected graph and let $\delta(G)$ be the minimum degree of $G$. Then,
$\beta_E(G) \geq 1 + \lceil log_2 \delta(G) \rceil$.\end{thm}
\begin{proof} Suppose the contrary, that there exists edge metric generator $S=\{w_1,w_2,...,w_p\}$ with cardinality 
$p < 1 + \lceil log_2 \delta(G) \rceil$.  Vertex $w_1$ is incident to at least $\delta(G)$ edges.
Name them as $e_1$, ..., $e_{\delta(G)}$. Since $w_1$ is incident with $e_1$, ..., $e_{\delta(G)}$,
it is obvious that $d(e_1,w_1)=...=d(e_{\delta(G)},w_1)=0$.
By the definition of distances between vertices and edges, it is clear that for an arbitrary vertex $v \in V (G)$ there 
can be only two different distances to some set of incident edges. 
Then, for each $i$, $i=2,...,p$, distances $d(e_1,w_i)$, ..., $d(e_{\delta(G)},w_i)$ have only two different values,
so since $d(e_1,w_1)=...=d(e_{\delta(G)},w_1)=0$ there exists at most $2^{p-1}$ different edge metric representations
of edges $e_1$, ..., $e_{\delta(G)}$ with respect to $S$, so $\delta(G) \leq 2^{p-1}$. Next, because $p$ has integer value,  
it follows that $\lceil log_2 \delta(G) \rceil \leq p-1$ $\Rightarrow$  $p \geq 1+\lceil log_2 \delta(G) \rceil$,
which is in contradiction to starting assumption! Therefore, $\beta_E(G) \geq 1 + \lceil log_2 \delta(G) \rceil$.
 \end{proof} 

In case of regular graphs, the bound presented in Proposition \ref{edim2} is improved by one.

\begin{cor} Let $G$ be an $r$-regular graph. Then, $\beta_E(G) \geq 1 + \lceil log_2 r \rceil$.\end{cor}

Since $GP(n,k)$ is 3-regular graphs, and $\lceil log_2 3 \rceil =2$ then it holds next corollary.

\begin{cor} \label{deg} $\beta_E(GP(n,k)) \geq 3$.\end{cor}

\subsection{Exact value for $GP(n,1)$}

In this section, we are studding generalized Petersen graphs
$GP(n,1)$. The following theorem gives the exact value of edge metric dimension of generalized Petersen graphs
$GP(n,1)$.

\begin{thm}\label{gpn1}$\beta_E(GP(n,1))=3$\end{thm}
\begin{proof} Let $S=\{u_0, u_1, v_0\}$. \\
\underline{{\it Case 1.}}  $n = 2t$\\
Edge metric representations with respect to $S$ are:
\begin{equation*}r(u_i u_{i+1},S) = \begin{cases}
(0,0,1), \; i=0 \\
(i,i-1,i+1),\; 1 \le i \le t-1 \\
(t-1,t-1,t), \; i=t \\
(2t-1-i,2t-i,2t-i),\; t+1 \le i \le 2t-1 \\
\end{cases}.
\end{equation*}

\begin{equation*}r(u_i v_i,S) = \begin{cases}
(0,1,0), \; i=0 \\
(i,i-1,i),\; 1 \le i \le t \\
(2t-i,2t+1-i,2t-i),\; t+1 \le i \le 2t-1 \\
\end{cases}.
\end{equation*}

\begin{equation*}r(v_i v_{i+1},S) = \begin{cases}
(1,1,0), \; i=0 \\
(i+1,i,i),\; 1 \le i \le t-1 \\
(t,t,t-1), \; i=t \\
(2t-i,2t+1-i,2t-1-i),\; t+1 \le i \le 2t-2 \\
(1,2,0), \; i=2t-1 \\
\end{cases}.
\end{equation*}

Since all edge metric representations with respect to $S$ are
mutually different, then $S$ is an edge metric generator.
Since $|S|=3$ and from Corollary \ref{deg} it follows $\beta_E(GP(2t,1)) = 3$.

\underline{{\it Case 2.}}  $n = 2t+1$\\
Edge metric representations with respect to $S$ are:
\begin{equation*}r(u_i u_{i+1},S) = \begin{cases}
(0,0,1), \; i=0 \\
(i,i-1,i+1),\; 1 \le i \le t \\
(2t-i,2t+1-i,2t+1-i),\; t+1 \le i \le 2t \\
\end{cases}.
\end{equation*}

\begin{equation*}r(u_i v_i,S) = \begin{cases}
(0,1,0), \; i=0 \\
(i,i-1,i),\; 1 \le i \le t \\
(t,t,t), \; i=t+1 \\
(2t+1-i,2t+2-i,2t+1-i),\; t+2 \le i \le 2t \\
\end{cases}.
\end{equation*}

\begin{equation*}r(v_i v_{i+1},S) = \begin{cases}
(1,1,0), \; i=0 \\
(i+1,i,i),\; 1 \le i \le t \\
(2t+1-i,2t+2-i,2t-i),\; t+1 \le i \le 2t \\
\end{cases}.
\end{equation*}

Similarly as in Case 1, all edge metric representations with respect to $S$ are
mutually different, so $S$ is an edge metric generator.
Having in mind that $|S|=3$ and from Corollary \ref{deg} it follows $\beta_E(GP(2t+1,1)) = 3$.
\end{proof}

In \cite{kel16}, the authors came to an interesting question of relation between metric dimension and edge metric dimension
of some graphs (called {\em realization question}), with the conclusion that it is possible to find all three cases, i.e. graphs $G$ such that $\beta_E(G) = \beta(G)$, 
$\beta_E(G) > \beta(G)$ or $\beta_E(G) < \beta(G)$.
For $GP(n,1)$ there are only two cases, since, from \cite{cac07}, it follows $\beta(GP(n,1))= \begin{cases} 
2, n \;is \; odd \\
3, n \;is \; even \\
\end{cases}$, so for $n=2t$ it holds $\beta_E(GP(n,1)) = \beta(GP(n,1))=3$,
while for $n=2t+1$ it holds $3=\beta_E(GP(n,1)) > \beta(GP(n,1))=2$.

Another interesting discussion is comparison between $\beta_E(GP(n,1))$ and $\beta_E'(GP(n,1))$. From \cite{nas18}, it follows that $\beta_E'(GP(n,1))=3$, which matches $\beta_E(GP(n,1))=3$ from Theorem~\ref{gpn1}.  

\subsection{Exact value for $GP(n,2)$}

Analyze in this section is focused on edge metric dimension of generalized Petersen graphs
$GP(n,2)$.   
Next theorem determines the exact value of the edge metric dimension for such graphs.

\begin{thm}\label{gpn2}$\beta_E(GP(n,2))=\begin{cases}
3, \;\; n=8 \vee n \ge 10 \\
4, \;\; n \in \{5,6,7,9\}
\end{cases}$\end{thm}
\begin{proof} In the case of $n=4t$, $t \geq 4$, let $S=\{u_0, v_3, v_{2t+3}\}$.
All edge metric representations with respect to $S$ are given in Table \ref{tabn2t0}.
The first column is related to edge $e \in E(GP(4t,2))$, the second column presents its
edge metric representation $r(e)$, while the last column gives the condition in which statement in the second column is true. 
As it can be observed from Table \ref{tabn2t0}, all edge metric representations with respect to $S$ are
mutually different, so $S$ is an edge metric generator for $GP(4t,2)$. Having in mind that $|S|=3$ 
and from Corollary \ref{deg} it follows that for $t \geq 4$, $\beta_E(GP(4t,2)) = 3$ holds.

If $n=4t+1$, $t \geq 4$, then let $S=\{u_0,  v_{2t-5}, v_{2t-4}\}$.
All edge metric representations with respect to $S$ are given in Table \ref{tabn2t1}.
As it can be seen in Table \ref{tabn2t1} all edge metric representations with respect to $S$ are
mutually different, so $S$ is a edge metric generator for $GP(4t+1,2)$. Again, having in mind that $|S|=3$ 
and from Corollary \ref{deg} it follows that for $t \geq 4$, $\beta_E(GP(4t+1,2)) = 3$ holds.

For $t \geq 4$ in cases when $n=4t+2$ or $n=4t+3$ let us define $S=\{u_0,  v_{2t-2}, v_{2t-1}\}$.
All edge metric representations of $GP(4t+2,2)$ and $GP(4t+3,2)$, with respect to $S$, 
are given in Table \ref{tabn2t2} and Table \ref{tabn2t3}, respectively.
It can be seen in Table \ref{tabn2t2} that all edge metric representations of $GP(4t+2,2)$, with respect to $S$, are
mutually different, so $S$ is a edge metric generator for $GP(4t+2,2)$. Again, having in mind that $|S|=3$ 
and from Corollary \ref{deg} it follows that for $t \geq 4$, $\beta_E(GP(4t+2,2)) = 3$ holds.
The same conclusion can be drawn for $GP(4t+3,2)$, since all its edge metric representations
presented in Table \ref{tabn2t3} are also mutually different, so for $t \geq 4$, $\beta_E(GP(4t+3,2)) = 3$ holds.

For the remaining cases when $n \leq 15$, edge metric dimension of $GP(n,2)$ 
is found by the total enumeration technique, and it is presented in Table \ref{tabgpn2}, together with the corresponding 
edge metric bases. It should be stated that edge metric dimension is equal to 3, except in cases for $n \in \{5,6,7,9\}$,
when it is equal to 4.
\end{proof}

For $GP(n,2)$ there are only two cases for realization question. From \cite{jav08} it follows that $\beta(GP(n,2))=3$, so for $n \notin \{5,6,7,9\}$ edge metric dimension of $GP(n,2)$
is equal to its metric dimension. Only in cases when $n \in \{5,6,7,9\}$, it holds $4=\beta_E(GP(n,2)) > \beta(GP(n,2))=3$.

Another interesting discussion is comparison between $\beta_E(GP(n,2))$ and $\beta_E'(GP(n,2))$. 
In contrast to $GP(n,1)$, values of $\beta_E(GP(n,2))$ and $\beta_E'(GP(n,2))$ sometimes differ. 
For example, $\beta_E(GP(9,2))=4$ while $\beta_E'(GP(n,2))=3$. 

\begin{table}
\caption{Edge metric representations for $GP(4t,2)$} \label{tabn2t0}
 \tiny
\begin{center} 
\begin{tabular}{|c|c|c|}
\hline
 $e$ & $r(e)$ & {\it Condition} \\
 \hline
$u_{2i} u_{2i+1}$ & $(0,2,t)$ & $i=0$ \\
& $(2,1,t+1)$ & $i=1$ \\
& $(i+2,i,t+2-i)$ & $2 \le i \le t$ \\
& $(2t+2-i,2t+2-i,i-t)$ & $t+1 \le i \le 2t-3$ \\
& $(3,4,t-2)$ & $i=2t-2$ \\
& $(1,3,t-1)$ & $i=2t-1$ \\
 \hline
$u_{2i+1} u_{2i+2}$ & $(1,2,t)$ & $i=0$ \\
& $(3,1,t+1)$ & $i=1$ \\
& $(i+3,i,t+2-i)$ & $2 \le i \le t-1$ \\
& $(t+1,t,2)$ & $i=t$ \\
& $(2t+1-i,2t+2-i,i-t)$ & $t+1 \le i \le 2t-3$ \\
& $(2,4,t-2)$ & $i=2t-2$ \\
& $(0,3,t-1)$ & $i=2t-1$ \\
 \hline
$u_{2i} v_{2i}$ & $(0,3,t)$ & $i=0$ \\
& $(2,2,t+1)$ & $i=1$ \\
& $(i+1,i,t+3-i)$ & $2 \le i \le t$ \\
& $(t,t+1,2)$ & $i=t+1$ \\
& $(2t+1-i,2t+3-i,i-t)$ & $t+2 \le i \le 2t-1$ \\
 \hline
$u_{2i+1} v_{2i+1}$ & $(1,1,t-1)$ & $i=0$ \\
& $(i+2,i-1,t+1-i)$ & $1 \le i \le t-1$ \\
& $(t+1,t-1,1)$ & $i=t$ \\
& $(2t+1-i,2t+1-i,i-t-1)$ & $t+1 \le i \le 2t-2$ \\
& $(1,2,t-2)$ & $i=2t-1$ \\
 \hline
$v_{2i} v_{2i+2}$ & $(1,3,t+1)$ & $i=0$ \\
& $(2,3,t+2)$ & $i=1$ \\
& $(i+1,i+1,t+3-i)$ & $2 \le i \le t-1$ \\
& $(t,t+1,3)$ & $i=t$ \\
& $(t-1,t+2,3)$ & $i=t+1$ \\
& $(2t-i,2t+3-i,i+1-t)$ & $t+2 \le i \le 2t-2$ \\
& $(1,4,t)$ & $i=2t-1$ \\
 \hline
$v_{2i+1} v_{2i+3}$ & $(2,0,t-1)$ & $i=0$ \\
& $(i+2,i-1,t-i)$ & $1 \le i \le t-1$ \\
& $(t,t-1,0)$ & $i=t$ \\
& $(2t-i,2t-i,i-t-1)$ & $t+1 \le i \le 2t-2$ \\
& $(2,1,t-2)$ & $i=2t-1$ \\
 \hline
\end{tabular}
\end{center}
\end{table}

\begin{table}
\caption{Edge metric representations for $GP(4t+1,2)$} \label{tabn2t1}
 \tiny
\begin{center}
\begin{tabular}{|c|c|c|}
\hline
 $e$ & $r(e)$ & {\it Condition} \\
 \hline $u_{2i} u_{2i+1}$ & $(0,t-2,t-1)$ & $i=0$ \\
& $(2,t-3,t-2)$ & $i=1$ \\
& $(i+2,t-2-i,t-1-i)$ & $2 \le i \le t-3$ \\
& $(i+2,i+4-t,i+3-t)$ & $t-2 \le i \le t$ \\
& $(2t+2-i,i+4-t,i+3-t)$ & $t+1 \le i \le 2t-3$ \\
& $(4t-2i, 3t-1-i, 3t-1-i)$ & $2t-2 \le i \le 2t$ \\  
 \hline
$u_{2i+1} u_{2i+2}$ & $(1,t-2,t-2)$ & $i=0$ \\
& $(3,t-3,t-3)$ & $i=1$ \\
& $(i+3,t-2-i,t-2-i)$ & $2 \le i \le t-3$ \\
& $(t+1,2,2)$ & $i=t-2$ \\
& $(t+2,3,3)$ & $i=t-1$ \\
& $(2t+2-i,i+4-t,i+4-t)$ & $t \le i \le 2t-3$ \\
& $(3,t,t+1)$ & $i=2t-2$ \\
& $(1,t-1,t)$ & $i=2t-1$ \\ 
 \hline
$u_{2i} v_{2i}$ & $(0,t-1,t-2)$ & $i=0$ \\
& $(2,t-2,t-3)$ & $i=1$ \\
& $(i+1,t-1-i,t-2-i)$ & $2 \le i \le t-3$ \\
& $(i+1,i+4-t,i+2-t)$ & $t-2 \le i \le t$ \\ 
& $(2t+2-i,i+4-t,i+2-t)$ & $t+1 \le i \le 2t-3$ \\
& $(4,t,t)$ & $i=2t-2$ \\
& $(3,t-1,t+1)$ & $i=2t-1$ \\
& $(1,t-2,t)$ & $i=2t$ \\
 \hline
$u_{2i+1} v_{2i+1}$ & $(1,t-3,t-1)$ & $i=0$ \\
& $(i+2,t-3-i,t-1-i)$ & $1 \le i \le t-3$ \\
& $(t,1,2)$ & $i=t-2$ \\
& $(t+1,2,3)$ & $i=t-1$ \\
& $(2t+1-i,i+3-t,i+4-t)$ & $t \le i \le 2t-3$ \\ 
& $(3,t+1,t)$ & $i=2t-2$ \\
& $(2,t,t-1)$ & $i=2t-1$ \\
 \hline
$v_{2i} v_{2i+2}$ & $(i+1,t-1-i,t-3-i)$ & $0 \le i \le t-4$ \\
& $(t-2,3,0)$ & $i=t-3$ \\
& $(t-1,3,0)$ & $i=t-2$ \\
& $(t,4,1)$ & $i=t-1$ \\
& $(2t+1-i,i+5-t,i+2-t)$ & $t \le i \le 2t-4$ \\
& $(2t+1-i, 3t-3-i, i+2-t)$ & $2t-3 \le i \le 2t-1$\\ 
& $(2,t-3,t)$ & $i=2t$ \\
 \hline
$v_{2i+1} v_{2i+3}$ & $(2,t-4,t-1)$ & $i=0$ \\
& $(i+2,t-4-i,t-1-i)$ & $1 \le i \le t-4$ \\
& $(t-1,0,3)$ & $i=t-3$ \\
& $(t,1,3)$ & $i=t-2$ \\
& $(2t-i,i+3-t,i+5-t)$ & $t-1 \le i \le 2t-4$ \\
& $(3,t,t)$ & $i=2t-3$ \\
& $(2,t+1,t-1)$ & $i=2t-2$ \\
& $(1,t,t-2)$ & $i=2t-1$ \\
 \hline
\end{tabular}
\end{center}
\end{table}

\begin{table}
\caption{Edge metric representations for $GP(4t+2,2)$} \label{tabn2t2}
 \tiny
\begin{center}
\begin{tabular}{|c|c|c|}
\hline
 $e$ & $r(e)$ & {\it Condition} \\
 \hline $u_{2i} u_{2i+1}$ & $(0,t,t+1)$ & $i=0$ \\
& $(2,t-1,t)$ & $i=1$ \\
& $(i+2,t-i,t+1-i)$ & $2 \le i \le t-1$ \\
& $(t+2,2,1)$ & $i=t$ \\
& $(2t+3-i,i+2-t,i+1-t)$ & $t+1 \le i \le 2t-2$ \\
& $(3,t+1,t)$ & $i=2t-1$ \\
& $(1,t+1,t+1)$ & $i=2t$ \\
 \hline
$u_{2i+1} u_{2i+2}$ & $(1,t,t)$ & $i=0$ \\
& $(3,t-1,t-1)$ & $i=1$ \\
& $(i+3,t-i,t-i)$ & $2 \le i \le t-1$ \\
& $(2t+2-i,i+2-t,i+2-t)$ & $t \le i \le 2t-2$ \\
& $(2,t+1,t+1)$ & $i=2t-1$ \\
& $(0,t+1,t+1)$ & $i=2t$ \\
 \hline
$u_{2i} v_{2i}$ & $(0,t+1,t)$ & $i=0$ \\
& $(2,t,t-1)$ & $i=1$ \\
& $(i+1,t+1-i,t-i)$ & $2 \le i \le t-1$ \\
& $(t+1,2,0)$ & $i=t$ \\
& $(2t+2-i,i+2-t,i-t)$ & $t+1 \le i \le 2t$ \\
 \hline
$u_{2i+1} v_{2i+1}$ & $(1,t-1,t+1)$ & $i=0$ \\
& $(i+2,t-1-i,t+1-i)$ & $1 \le i \le t-1$ \\
& $(2t+2-i,i+1-t,i+2-t)$ & $t \le i \le 2t-1$ \\ 
& $(1,t,t+2)$ & $i=2t$ \\
 \hline
$v_{2i} v_{2i+2}$ & $(i+1,t+1-i,t-1-i)$ & $0 \le i \le t-2$ \\
& $(t,3,0)$ & $i=t-1$ \\
& $(2t+1-i,i+3-t,i-t)$ & $t \le i \le 2t-1$ \\
& $(1,t+2,t)$ & $i=2t$ \\
 \hline
$v_{2i+1} v_{2i+3}$ & $(i+2,t-2-i,t+1-i)$ & $0 \le i \le t-2$ \\
& $(t+1,0,3)$ & $i=t-1$ \\
& $(2t+1-i,i+1-t,i+3-t)$ & $t \le i \le 2t-1$ \\
& $(2,t-1,t+2)$ & $i=2t$ \\
 \hline
\end{tabular}
\end{center}
\end{table}

\begin{table}
\caption{Edge metric representations for $GP(4t+3,2)$} \label{tabn2t3}
 \tiny
\begin{center}
\begin{tabular}{|c|c|c|}
\hline
 $e$ & $r(e)$ & {\it Condition} \\
 \hline $u_{2i} u_{2i+1}$ & $(0,t,t+1)$ & $i=0$ \\
& $(2,t-1,t)$ & $i=1$ \\
& $(i+2,t-i,t+1-i)$ & $2 \le i \le t-1$ \\
& $(t+2,2,1)$ & $i=t$ \\
& $(2t+3-i,i+2-t,i+1-t)$ & $t+1 \le i \le 2t-1$ \\
& $(2,t+2,t+1)$ & $i=2t$ \\
& $(0,t+1,t+1)$ & $i=2t+1$ \\
 \hline
$u_{2i+1} u_{2i+2}$ & $(1,t,t)$ & $i=0$ \\
& $(3,t-1,t-1)$ & $i=1$ \\
& $(i+3,t-i,t-i)$ & $2 \le i \le t-1$ \\
& $(2t+3-i,i+2-t,i+2-t)$ & $t \le i \le 2t-2$ \\
& $(3,t+1,t+1)$ & $i=2t-1$ \\
& $(1,t+1,t+2)$ & $i=2t$ \\
 \hline
$u_{2i} v_{2i}$ & $(0,t+1,t)$ & $i=0$ \\
& $(2,t,t-1)$ & $i=1$ \\
& $(i+1,t+1-i,t-i)$ & $2 \le i \le t-1$ \\
& $(t+1,2,0)$ & $i=t$ \\
& $(2t+3-i,i+2-t,i-t)$ & $t+1 \le i \le 2t-1$ \\
& $(3,t+1,t)$ & $i=2t$ \\
& $(1,t,t+1)$ & $i=2t+1$ \\
 \hline
$u_{2i+1} v_{2i+1}$ & $(1,t-1,t+1)$ & $i=0$ \\
& $(i+2,t-1-i,t+1-i)$ & $1 \le i \le t-1$ \\
& $(2t+2-i,i+1-t,i+2-t)$ & $t \le i \le 2t-1$ \\ 
& $(2,t+1,t+1)$ & $i=2t$ \\
 \hline
$v_{2i} v_{2i+2}$ & $(i+1,t+1-i,t-1-i)$ & $0 \le i \le t-2$ \\
& $(t,3,0)$ & $i=t-1$ \\
& $(t+1,3,0)$ & $i=t$ \\
& $(2t+2-i,i+3-t,i-t)$ & $t+1 \le i \le 2t-2$ \\
& $(3,t+1,t-1)$ & $i=2t-1$ \\
& $(2,t,t)$ & $i=2t$ \\
& $(2,t-1,t+1)$ & $i=2t+1$ \\
 \hline
$v_{2i+1} v_{2i+3}$ & $(i+2,t-2-i,t+1-i)$ & $0 \le i \le t-2$ \\
& $(t+1,0,3)$ & $i=t-1$ \\
& $(2t+1-i,i+1-t,i+3-t)$ & $t \le i \le 2t-2$ \\
& $(2,t,t+1)$ & $i=2t-1$ \\
& $(1,t+1,t)$ & $i=2t$ \\
 \hline
\end{tabular}
\end{center}
\end{table}

\begin{table}
\caption{Edge resolving bases of $GP(n,2)$} \label{tabgpn2}
\begin{center}
\begin{tabular}{|c|c|c|}
\hline
 $n$ & {\it basis } & $\beta_E(GP(n,2))$ \\
 \hline
$5$ & $\{u_0,u_1,u_3,v_3\}$ & 4 \\   
 $6$ & $\{u_0,u_1,u_2, u_3\}$ & 4 \\   
 $7$ & $\{u_0,u_1,u_4,v_2\}$ & 4 \\   
 $8$ & $\{u_0,u_2,v_4\}$ & 3 \\   
 $9$ & $\{u_0,u_1,u_2,v_5\}$ & 4 \\   
 $10$ & $\{u_0,u_3,v_6\}$ & 3 \\   
 $11$ & $\{u_0,u_3,v_4\}$ & 3 \\   
 $12$ & $\{u_0,u_3,v_4\}$ & 3 \\   
 $13$ & $\{u_0,v_3,v_4\}$ & 3 \\   
 $14$ & $\{u_0,u_4,v_1\}$ & 3 \\   
 $15$ & $\{u_0,u_5,v_1\}$ & 3 \\   
 \hline
$n=4t \wedge t \ge 4$ & $\{u_0, v_3, v_{2t+3}\}$ & 3 \\
$n=4t+1 \wedge t \ge 4$ & $\{u_0,  v_{2t-5}, v_{2t-4}\}$ & 3 \\
$(n=4t+2 \vee n=4t+3) \wedge t \ge 4$ & $\{u_0,  v_{2t-2}, v_{2t-1}\}$ & 3 \\
 \hline
 \end{tabular}
\end{center}
\end{table}

\section{Conclusions}
In this article, the recently introduced edge metric dimension problem is
considered.
Exact formulae for generalized Petersen graphs $GP(n,1)$ and $GP(n,2)$ are stated and proved. Moreover, the lower bound for 3-regular graphs, which holds for all generalized Petersen graphs, is given.

Possible future research could be finding the edge metric dimension of some other challenging classes of graphs.
Another research direction is construction of some metaheuristic approach for solving an edge metric dimension problem.


\begin{thebibliography}{10}
\expandafter\ifx\csname url\endcsname\relax
  \def\url#1{\texttt{#1}}\fi
\expandafter\ifx\csname urlprefix\endcsname\relax\def\urlprefix{URL }\fi
\expandafter\ifx\csname href\endcsname\relax
  \def\href#1#2{#2} \def\path#1{#1}\fi

\bibitem{metd1}
P.~J. Slater, Leaves of trees, Congr. Numer 14~(549-559) (1975) 37.

\bibitem{metd2}
F.~Harary, R.~Melter, On the metric dimension of a graph, Ars Combin
  2~(191-195) (1976) 1.

\bibitem{kel16}
A.~Kelenc, N.~Tratnik, I.~G. Yero, Uniquely identifying the edges of a graph:
  the edge metric dimension, Discrete Applied Mathematics 251 (2018) 204--220.

\bibitem{yer16}
I.~G. Yero, Vertices, edges, distances and metric dimension in graphs,
  Electronic Notes in Discrete Mathematics 55 (2016) 191--194.

\bibitem{pet18}
I.~Peterin, I.~G. Yero, Edge metric dimension of some graph operations, arXiv
  preprint arXiv:1809.08900.

\bibitem{zub16a}
N.~Zubrilina, On the edge dimension of a graph, Discrete Mathematics 341~(7)
  (2018) 2083--2088.

\bibitem{zub16b}
N.~Zubrilina, On the edge metric dimension for the random graph, arXiv preprint
  arXiv:1612.06936.

\bibitem{eps15}
L.~Epstein, A.~Levin, G.~J. Woeginger, The (weighted) metric dimension of
  graphs: Hard and easy cases, Algorithmica 72~(4) (2015) 1130--1171.

\bibitem{liu18}
J.-B. Liu, Z.~Zahid, R.~Nasir, W.~Nazeer, Edge version of metric dimension and
  doubly resolving sets of the necklace graph, Mathematics 6~(11) (2018) 243.

\bibitem{gpprvi}
H.~Coxeter, Self-dual configurations and regular graphs, Bull. Amer. Math. Soc
  56 (1950) 413--455.

\bibitem{cac07}
J.~C{\'a}ceres, C.~Hernando, M.~Mora, I.~M. Pelayo, M.~L. Puertas, C.~Seara,
  D.~R. Wood, On the metric dimension of Cartesian products of graphs, SIAM
  Journal on Discrete Mathematics 21~(2) (2007) 423--441.

\bibitem{jav08}
I.~Javaid, M.~T. Rahim, K.~Ali, Families of regular graphs with constant metric
  dimension, Utilitas Mathematica 75 (2008) 21--34.

\bibitem{imr14}
M.~Imran, A.~Q. Baig, M.~K. Shafiq, I.~Tomescu, On metric dimension of
  generalized Petersen graphs p(n,3), Ars Comb. 117 (2014) 113--130.

\bibitem{wan17}
X.~Wang, All double generalized Petersen graphs are Hamiltonian, Discrete
  Mathematics 340~(12) (2017) 3016--3019.

\bibitem{bal17}
T.~Ball, R.~W. Bell, J.~Guzman, M.~Hanson-Colvin, N.~Schonsheck, On the cop
  number of generalized Petersen graphs, Discrete Mathematics 340~(6) (2017)
  1381--1388.

\bibitem{dan16}
S.~Dantas, C.~M. de~Figueiredo, G.~Mazzuoccolo, M.~Preissmann, V.~F.
  Dos~Santos, D.~Sasaki, On the total coloring of generalized Petersen graphs,
  Discrete Mathematics 339~(5) (2016) 1471--1475.

\bibitem{nas18}
R.~Nasir, S.~Zafar, Z.~Zahid,
  \href{https://www.researchgate.net/publication/322634658_Edge_metric_dimension_of_graphs}{Edge
  metric dimension of graphs}, Ars Combinatoria, in press.
\newline \small{\url{https://www.researchgate.net/publication/322634658_Edge_metric_dimension_of_graphs}}

\end{thebibliography}
\end{document}